\documentclass[11pt,a4paper]{article}
\usepackage{inputenc}
\usepackage{amssymb}
\usepackage{amsmath}
\usepackage{amsthm}
\usepackage[square,sort,comma,numbers]{natbib}
\usepackage{extarrows}
\usepackage{setspace}
\usepackage{color}
\usepackage{enumerate}
\usepackage{hyperref}
\usepackage{graphicx, subfig}
\usepackage{natbib}
\usepackage{setspace}
\theoremstyle{definition}
\newtheorem{Def}{Definition}[section]
\newtheorem{Exa}[Def]{Examples}

\theoremstyle{plain}
\newtheorem{Thm}[Def]{Theorem}
\newtheorem{Lem}[Def]{Lemma}
\newtheorem{Pro}[Def]{Proposition}

\usepackage[left=2.6cm, right=2.6cm, top=3cm, bottom=3cm]{geometry}
\usepackage[all]{xy}
\usepackage{bm}

\def\Box{{\bf Box}}
\def\IN{\mathbb N}
\def\IR{\mathbb R}

\def\O{\mathcal O}

\def\supp{\textup{supp}}
\def\diam{\textup{diam}}

\def\acton{\curvearrowright}

\author{Jintao Deng, \and Liang Guo,\and Qin Wang,\and Yazhou Zhang}
\title{Coarse embeddings at infinity and generalized expanders at infinity*\footnotetext{ Supported in part by NSFC (No.   11831006, 12171156).}}
\date{}
\begin{document}
	\maketitle
	
	\begin{abstract}
We introduce a notion of {\em coarse embedding at infinity} into Hilbert space for metric spaces, which is a weakening of the notion of fibred coarse embedding and a far generalization of Gromov's concept of coarse embedding. It turns out that a residually finite group admits a coarse embedding into Hilbert space if and only if one (or equivalently, every) box space of the group admits a coarse embedding at infinity into Hilbert space. Moreover, we introduce a concept of {\em generalized expander at infinity} and show that it is an obstruction to coarse embeddability at infinity.
	\end{abstract}
	
	\section{Introduction}
	The concept of coarse embeding was introduced by  M. Gromov \cite{Gromov1993}. In \cite{Yu2000}, G. Yu showed that the coarse Baum--Connes conjecture holds for metric spaces which are coarsely embeddable into Hilbert space. Recall that a metric space $(X,d_X)$ admits a coarse embedding into Hilbert space if there exists a map $f:X\to H$ and  two non-decreasing functions $\rho_-,\rho_+:[0,\infty)\to [0,\infty)$ with $\lim\limits_{t\rightarrow \infty}\rho_-(t)=\infty$ such that for all $x,y\in X$
	$$\rho_-(d(x,y))\leq \|f(x)-f(y)\|\leq \rho_+(d(x,y)).$$
	
	It is well known that expanders are not coarsely embeddable into Hilbert space. In \cite{CWY2013}, X. Chen, Q. Wang and G. Yu introduced a concept of fibred coarse embedding into Hilbert space for metric spaces. It turns out that a large class of expanders admit a fibred coarse embedding into  Hilbert space, and that the maximal coarse Baum--Connes conjecture holds for bounded geometry spaces which admit a fibred coarse embedding into Hilbert space.

	In this paper, we shall introduce a notion of {\em coarse embedding at infinity} into Hilbert space (Definition \ref{CEinf}), which is a weakening of the notion of fibred coarse embedding into Hilbert space. Roughly speaking, a metric space $X$ admits a coarse embedding at infinity if large bounded subsets  of $X$ which are far away towards infinity can be coarsely embedded into Hilbert space within common distortion controlling functions.  We show that 

\begin{Thm} (Theorem \ref{CEiffCEinf})
The box space of a finitely generated residually finite group associated with any filtration admits
	a coarse embedding at infinity into Hilbert space  if and only if the group is coarsely embeddable into Hilbert space.
\end{Thm}

In this theorem, the role of Hilbert space can be replaced by any class of metric spaces which is closed under taking direct sums and ultra-products, e.g., the class of $L^p$ spaces with $1\leq p<\infty$, the class of uniformly convex Banach spaces, etc.. This fits into the following spectrum which links analytic properties of a group $\Gamma$ to the  geometric properties of its box spaces $\Box(\Gamma)$.
	\begin{equation*}\begin{split}
	\Gamma\mbox{ is amenable}&\iff\Box(\Gamma)\mbox{ has Yu's Property A,}\\
	\Gamma\mbox{ is a-T-menable}&\iff\Box(\Gamma)\mbox{ admits a fibred coarse embedding into Hilbert space,}\\
	\Gamma\mbox{ has Property(T)}&\iff\Box(\Gamma)\mbox{ has geometric Property(T),}\\
	\Gamma\mbox{ has Property A}&\iff\Box(\Gamma)\mbox{ has Property A at infinity,}\\
	\Gamma\txt{\mbox{ is coarsely embeddable} \\ \mbox{ into Hilbert space}}&\iff\Box(\Gamma)\mbox{ admits a coarse embedding at infinity into Hilbert space} \\
	\Gamma\txt{\mbox{ is coarsely embeddable} \\ \mbox{ into an $L^p$ space}}&\iff\Box(\Gamma)\mbox{ admits a coarse embedding at infinity into an $L^p$ space}
	\end{split}
	\end{equation*}
	where $\Box(\Gamma)$ is the box space of $\Gamma$ associated with a filtration of $\Gamma$.
	
	The first equivalence was due to E. Guentner (refer to \cite{Roe2003}), where Property A is a geometric version of amenability introduced by G. Yu \cite{Yu2000} which guarantees the coarse embeddability into Hilbert space. The second equivalence was proved by X. Chen, Q. Wang and X. Wang \cite{CWW2013} to provide a tool for the study of the maximal Baum--Connes conjecture \cite{CWY2013}. The
	third equivalence is due to R. Willett and G. Yu \cite{RYI2012} on higher index theory for certain expander graphs and Gromov monster groups.  The forth equivalence was proved by T. Pillon \cite{T2018}. We will prove the fifth (Theorem \ref{CEiffCEinf}), and the last conclusion can be proved in a similar way (cf. \cite{Arnt2016}).  Pillon's notion of Property A at infinity \cite{T2018} implies coarse embeddability at infinity into Hilbert space. We also show that the coarse disjoint union of a sequence of group extensions with ``$CE_\infty-by-A_\infty$" structure implies ``$CE_\infty$".

	The obstructions to the coarse embeddability into Hilbert space have been widely studied (c.f.\cite{AT2015, Arzhantseva2019, Gromov03, Tessera, NY2012, Osajda2021}). M. I. Ostrovskii \cite{Ostro09} and R. Tessera \cite{Tessera} independently showed that a metric space X does not coarsely embed into Hilbert space if and only if $X$ contains a sequence of finite subsets of $X$, called generalized expanders. In this paper, we study the obstructions to the coarse embeddability at infinity into Hilbert space.  We introduce a notion of {\em generalized expanders at infinity} which can be characterized by using a Poincar\'{e} inequality, and then prove the following result.
	
\begin{Thm}{(see Theorem 3.2)}
	A  metric space $X$ does not coarsely embed at infinity into Hilbert space  if and only if $X$ coarsely contains a generalized expanders at infinity.
\end{Thm}	
	
The well-known concrete obstructions to the coarse embeddability into Hilbert space include expanders, relative expanders \cite{AT2015} and asymptotic expanders \cite{asymptoticexpander}. In this paper, we also introduce a notion of containment of an expander at infinity  to demonstrate a special case of generalized expanders at infinity.

		\begin{Thm}{(see Theorem 3.8)}
		Let $\{X_n\}_{n\in \mathbb{N}}$ be a sequence of connected finite graphs. If $\{X_n\}_{n\in \mathbb{N}}$ contains an expander at infinity, then the coarse disjoint union $X=\sqcup_{n\in \mathbb{N}} X_n$ is a generalized expander at infinity, and thus does not coarsely embed at infinity into Hilbert space.
	\end{Thm}
	

	The paper is organized as follows. In Section 2, we introduce the notion of coarse embedding at infinity into Hilbert space and discuss several cases to which this notion applies, such as box spaces and warped cones. In section 3, we introduce a notion of generalized expanders at infinity, and show that it is an obstruction to coarse embedding at infinity. Morover, we introduce a notion of expanders at infinity as a special case of generalized expander at infinity. In section 4, we introduce the notion of $p$-convex metric and  $(\mathcal{M},p)$-generalized expanders at infinity for a class of metric spaces $\mathcal{M}$.

	\section{Coarse embeddings at infinity}
	
	In this section, we shall first introduce the notion of coarse embeddability at infinity, and then show that this property is preserved under certain group extensions.
	
	Let $X$, $Y$ be metric spaces. A map $f:X\to Y$ is called a \emph{coarse embedding} if there exist non-decreasing functions $\rho_{\pm}:\mathbb{R}_+\to\mathbb{R}_+$ with $\displaystyle \lim_{t\to\infty}\rho_{\pm}(t)=\infty$ such that
	$$\rho_-(d_X(x,y))\leq d_Y(f(x),f(y))\leq\rho_+(d_X(x,y))$$
	for all $x,y\in X$. We say that a metric space $X$ is \emph{coarsely embeddable} if there exists a coarse embedding of $X$ into Hilbert space.
	
	Now we recall the notion of coarse embeddability for a sequence of metric spaces. A sequence of metric spaces $\{X_i\}_{i\in I}$ is said to admit a coarse embedding into a Hilbert space $H$ if there exist non-decreasing functions $\rho_{\pm}:\mathbb{R}_+\to\mathbb{R}_+$ with $\displaystyle \lim_{t\to\infty}\rho_{\pm}(t)=\infty$, such that for each $i\in I$, there exists $f_i:X_i\to H$ satisfying
	$$\rho_-(d_X(x,y))\leq \|f_i(x)-f_i(y)\|\leq\rho_+(d_X(x,y))$$
	for all $x,y\in X_i$.
	
	We now introduce a generalized concept of coarse embeddability.
	\begin{Def}\label{CEinf}
		A metric space $(X,d)$ is said to admit a {\em coarse embedding at infinity} into Hilbert space  if there exist two non-decreasing functions $\rho_+,\rho_-:\mathbb R_+\to\mathbb R_+$ with $\lim\limits_{t\to\infty}\rho_{\pm}(t)=\infty,$ such that for each $R>0$, there exists a bounded subset $K_R\subseteq X$, such that for any bounded subset $C\subset X\backslash K_R$ with $\diam(C)\leq R$, there exists a coarse embedding  $f_C:C\to H$ into a Hilbert space $H$ satisfying
		$$\rho_-(d(x,y))\leq\|f_C(x)-f_C(y)\|\leq \rho_+(d(x,y))$$
		for all $x,y\in C$.
	\end{Def}
	We remark here that one can also consider coarse embedding at infinity into other metric spaces, such as $L^p$-spaces, uniformly convex Banach spaces, CAT($0$) spaces, etc..
	
	Recall that a coarse disjoint union of a sequences of finite space $(X_n,d_n)_{n\in\IN}$ is the disjoint union $X=\bigsqcup_{n\in\IN}X_n$ equipped with a metric $d$ such that the restriction of $d$ to each $X_n$ coincides with the original metric $d_n$ and $d(X_n,X_m)\to\infty$ as $n+m\to\infty$ and $n\ne m$.
	
     Recall that a metric space $\widetilde{Y}$ is a Galois covering of $Y$ if there exists a discrete group $\Gamma$ acting on $\widetilde{Y}$ freely and properly by isometries such that $Y=\widetilde{Y}/\Gamma$. Denote  $\pi:\widetilde{Y}\to Y$ to be the associated covering map. A sequence of Galois coverings $(\widetilde{X}_n)_{n\in\IN}$ of $(X_n)_{n\in\IN}$ is said to be \emph{asymptotic faithful} if for any $r>0$, there exists a integer $N$ such that for all $n\geq N$, the covering map $\pi_n$ is "$r$-isometric", that is, for all subsets $\widetilde{C}\subset \widetilde{X}_n$ of diameter less than $r$, the restriction of $\pi_n$ to $\widetilde{C}$ is an isometry onto $C=\pi_n(\widetilde{C})\subset X_n$.
	
	Let us recall the definition of box spaces of a finitely generated residually finite group.
	
	\begin{Def}
		Let $\Gamma$ be a finitely generated residually finite group. A filtration of $\Gamma$ is a nested sequence of finite index normal groups $\Gamma_1\supseteq\cdots\supseteq\Gamma_i\supseteq\cdots$ of $\Gamma$ such that $\bigcap_{i=1}^{\infty}\Gamma_i=\{e\}$. The box space of $\Gamma$ associated with $\{\Gamma_i\}$ is defined to be the coarse disjoint union $\sqcup_{i=1}^{\infty}\Gamma/\Gamma_i$, denoted by $\Box_{\{\Gamma_i\}}(\Gamma)$, or simply $\Box(\Gamma)$, where each quotient $\Gamma/\Gamma_n$ is endowed with the quotient metric from $\Gamma$.
	\end{Def}
	Box spaces provide a large class of metric spaces which are coarsely embeddable at infinity.
	\begin{Thm}\label{CEiffCEinf}
		Let $\Gamma$ be a finitely generated residually finite group. Then the followings are equivalent.
		\begin{enumerate}[(1)]
			\item The group $\Gamma$ admits a coarse embedding into Hilbert space;
			\item There exists a filtration $\{\Gamma_i\}$ such that $\Box_{\{\Gamma_i\}}(\Gamma)$ is coarsely embeddable at infinity into Hilbert space;
			\item For any filtration $\{\Gamma_i\}$, $\Box_{\{\Gamma_i\}}(\Gamma)$ is coarsely embeddable at infinity into Hilbert space.
		\end{enumerate}
	\end{Thm}
	
	\begin{proof} (1)$\Rightarrow$(3). Let $f:\Gamma\to H$ be a coarse embedding of $X$ into Hilbert space $\mathcal{H}$. Assume that $\rho_+,\rho_-:\IR_+\to\IR_+$ are non-decreasing functions satisfying $\lim\limits_{t\to\infty}\rho_{\pm}(t)=\infty$ and
		$$\rho_-(d(g,h))\leq\|f(g)-f(h)\|\leq \rho_+(d(g,h))$$
		for any $g,h\in\Gamma$. For any filtration $\{\Gamma_i\}$, choose an element $m_i\in\Gamma_i$ for each $i\in\IN$ satisfying that
		$$l_{\Gamma}(m_i)=\min\{l_{\Gamma}(g)\mid g\in\Gamma_i\}.$$
		The existence of $m_i$ is guaranteed by the condition that $\Gamma_i$ has finite index in $\Gamma$. As $\bigcap_{i=1}^{\infty}\Gamma_i=\{e\}$, the sequence $\{l_{\Gamma}(m_i)\}_{i\in\IN}$ is non-decreasing and we have that $\lim\limits_{i\to\infty}l_{\Gamma}(m_i)=+\infty$.
		
		For any $R>0$, fix an $N\in\IN$ such that $l_{\Gamma}(m_N)>2R$. For all $\widetilde{C}\subset \Gamma$ with $\diam(C)\leq R$, we have the restriction of the quotient homomorphisms $\pi_i:\Gamma\to\Gamma/\Gamma_i$ to $\widetilde{C}$ is an isometry onto $C=\pi_i(\widetilde{C})$ for all $i>N$. Let $K_R=\sqcup_{i=1}^{N}\Gamma/\Gamma_i$, then for all $C\subset\Gamma\backslash K_R$ of $\diam(C)\leq R$, we can choose $\tilde C\subset\Gamma$ such that $\pi_i|_{\widetilde{C}}:\tilde C\to C$ is an isometric bijection. Define
		$$f_C=f\circ(\pi_i|_{\widetilde{C}})^{-1}:C\to H.$$
		As $\pi_i|_{\widetilde{C}}$ is an isometry and $f$ is a coarse map, then $f_C$ is also a coarse embedding with respect to $\rho_{\pm}$.
		
		(3)$\Rightarrow$(2). It is obvious thus omitted.
		
		(2)$\Rightarrow$(1). We will prove it using an equivalent characterization of coarse embeddability introduced in \cite[Proposition 3.2]{DGLY2002}. Recall that a locally finite metric space $X$ is said to admit a coarse embedding into Hilbert space if and only if there exist non-decreasing functions $\rho_+,\rho_-:\IR_+\to\IR_+$ satisfying $\lim\limits_{t\to\infty}\rho_{\pm}(t)=\infty$ such that, for every finite subspace $F\subset X$, there exists a map $f_F:F\to H$ satisfying
		$$\rho_-(d(x,y))\leq\|f_F(x)-f_F(y)\|\leq \rho_+(d(x,y))$$
		for all $x,y\in F$.
		
		Now, assume that $\Box_{\{\Gamma_i\}}(\Gamma)$ admits a coarse embedding at infinity into $ H$  for a certain filtration, it suffices to prove that any finite subsets of $\Gamma$ coarsely embed into $ H$ uniformly. Let $\rho_+,\rho_-$ be as in Definition \ref{CEinf}. For any bounded subset $F\subset\Gamma$, there exists $N_0>0$ such that the restriction of $\pi_i:\Gamma\to\Gamma/\Gamma_i$ to $F$ is an isometry for all $i>N_0$. Set $R=\diam(F)$. Since $\Box_{\{\Gamma_i\}}(\Gamma)$ admits a coarse embedding at infinity into $ H$ , there exists a bounded set $K_R\subset\Gamma$. Take a sufficiently large $N_1$ such that $K_R\cap \Gamma/\Gamma_i=\emptyset$ for all $i>N_1$. Let $N=\max\{N_0,N_1\}+1$. We have that $\pi_N(F)\cap K_R=\emptyset$ and $\diam(\pi_N(F))=\diam(F)=R$. As a result, there exists a coarse embedding $f_{\pi_N(F)}:\pi_N(F)\to H$ according to $\rho_{\pm}$. Define
		$$f_F=f_{\pi_N(F)}\circ\pi_N|_F:F\to H.$$
		It is obvious that $f_F$ is a coarse embedding with respect to the controlling functions $\rho_{\pm}$. Thus, $\Gamma$ admits a coarse embedding into $ H$ by the equivalent characterization.
	\end{proof}
	
	Following the argument above, we can generalize Theorem 2.3 to the following.
	\begin{Thm}\label{Galois cover}
		Let $X=\bigsqcup_{n\in\IN}X_n$ be the coarse disjoint union of a sequence of bounded metric spaces. If $(X_n)_{n\in \mathbb{N}}$ has a sequence of asymptotically faithful Galois coverings $(\widetilde{X}_n)_{n\in\IN}$ which admits a coarse embedding into Hilbert space, then $X$ admits a coarse embedding at infinity into Hilbert space.
	\end{Thm}
	
	We remark here that the coarse embeddability at infinity into Hilbert space  for a group is equivalent to the coarse embeddability into Hilbert space (c.f. \cite{DGLY2002}). Assume that $G$ is a residually finite group which admits no coarse embedding into Hilbert space. By Theorem \ref{CEiffCEinf}, there exists a box space of $G$ which admits no coarse embedding at infinity into Hilbert space. Thus there are metric spaces which are not coarsely embeddable at infinity into Hilbert space. It is worth noting that D. Osajda \cite{Osajda18} constructed a finitely
	generated, residually finite non-exact groups, whose box spaces admit a coarsely embedding at
	infinity into Hilbert space.
	
	Now, we will list some examples of metric spaces which admit  a coarse embedding at infinity into Hilbert space.
	\begin{Exa}[Fibred coarse embedding into Hilbert space \cite{CWY2013}]
		 A metric space $(X,d)$ is said to admit a fibred coarse embedding into Hilbert space if there exists
		\begin{enumerate}[(1)]
			\item a field of Hilbert spaces $( H_x)_{x\in X}$ over $X$;
			\item a section $s:X\to\sqcup_{x\in X_n} H_x$;
			\item two non-decreasing functions $\rho_{\pm}:\mathbb{R}_+\to \mathbb{R}_+$ with $\displaystyle\lim_{r\to\infty}\rho_{\pm}(r)=\infty$
		\end{enumerate}
		such that for any $r>0$, there exists a bounded subset $K\subset X$ for which there exists a "trivialization"
		$$t_C:( H_x)_{x\in C}\to C\times  H$$
		for each subset $C\subset X\backslash K$ of diameter less than $r$, such that the restriction of $t_C$ to the fiber $ H_x$ $(x\in C)$ is an isometry $t_x(z): H_z\to  H$, satisfying
		\begin{enumerate}[(1)]
			\item for any $z,z'\in C$,
			$$\rho_-(d(z,z'))\leq \|t_C(z)(s(z))-t_C(z')(s(z'))\|\leq \rho_+(d(z,z'));$$
			\item for any two subsets $C_1, C_2\subset X\backslash K$ of diameter less than $r$ with $C_1\cap C_2\ne\emptyset$, there exists an isometry $t_{C_1C_2}: H\to  H$ such that $t_{C_1}(z)\circ t^{-1}_{C_2}(z)= t_{C_1C_2}$ for all $z\in C_1\cap C_2$.
		\end{enumerate}
	\end{Exa}
	It is obvious that if a metric space $X$ admits a fibred coarse embedding into Hilbert space, then it admits a coarse embedding at infinity into Hilbert space.
	
	\begin{Exa}[Pillon's Property A at infinity \cite{T2018}]\label{Property A at infinity}
		The notion of Property A was introduced by Yu \cite{Yu2000} as a coarse analogue of amenability. We will work with the following characterization of Property A introduced in \cite{Tu2001}. A discrete metric space $X$ with bounded geometry has Property A if and only if for every $R>0$ and $\varepsilon>0$ there exists a map $\xi:X\to\ell^2(X),x\mapsto\xi_x$ and $S>0$ such that
		\begin{enumerate}[(1)]
			\item $\|\xi_x\|=1$ for all $x\in X$;
			\item $|1-\langle\xi_x-\xi_y\rangle|\leq\varepsilon$ for all $d(x,y)\leq R$;
			\item ${\rm supp}(\xi_x)\subset B_X(x,S)$ for all $x\in X$.
		\end{enumerate}
		Let $\{X_i\}_{i\in I}$ be a sequence of metric spaces. Then $\{X_i\}_{i\in I}$ is said to have Property A if $X_i$ has Property  A with uniform parameters $R,\varepsilon, S$ above.
		
		In \cite{T2018}, T. Pillon introduced a notion of coarse amenability at infinity. We will call it Property A at infinity in this paper. Let $X$ be a uniformly discrete metric space with bounded geometry. The space $X$ is said to have {\em Property A at infinity} if for all $R,\varepsilon> 0$, there exists $S\geq 0$ such that for any $L\geq 0$, there exist a finite subset $K_L\subset X$ with the property that all finite subsets $C\subset X\backslash K_L$ of diameter at most $L$ has Property A, i.e. the parameter $S$ does not depend on $L$.
		
		It was proved in \cite{NY2012} that a metric space $X$ with Property A is coarsely embeddable and the controlling functions $\rho_{\pm}$ are only determined by three parameters $R,\varepsilon$ and $S$ above. Thus it is clear that for a metric space with bounded geometry, if it has Property A at infinity, then is coarsely embeddable at infinity into Hilbert space.
	\end{Exa}
	
	
	We say that  a sequence of groups $(G_n)_{n \in \mathbb{N}}$ has uniformly finite generating sets if there exists a positive integer $N$ such that each group $G_n$ has a generating subset $S_n$ with $|S_n|\leq N$. Note that if the sequence $(G_n)_{n \in \mathbb{N}}$ has uniformly finite generating sets, then the coarse disjoint union $\bigsqcup G_n$ has bounded geometry. Let us now consider the coarse embeddings at infinity into Hilbert space  under extensions of groups.
	\begin{Pro}\label{extension}
		Let $(1\to N_n\to G_n\to Q_n\to 1)_{n\in\IN}$ be a sequence of extensions of finite groups with uniformly finite generating sets. Suppose that
		\begin{enumerate}[(1)]
			\item the coarse disjoint union $N=\bigsqcup_{n\in\IN}N_n$ with the induced metric from the word metrics of $(G_n)_{n\in\IN}$  admits a coarse embedding at infinity into Hilbert space;
			\item the coarse disjoint union $Q=\bigsqcup_{n\in\IN}Q_n$ with the quotient metrics has Property A at infinity,
		\end{enumerate}
		Then the coarse disjoint union $G=\bigsqcup_{n\in\IN}G_n$ admits a coarse embedding at infinity into Hilbert space.
	\end{Pro}
	
	To prove Proposition \ref{extension}, we need the following equivalent characterization of coarse embeddability \cite[Proposition 2.1]{DG2003}.
	
	\begin{Lem}\cite{DG2003}\label{coarse embeddable}
		A sequence of metric spaces $\{X_i\}_{i\in I}$ is coarsely embeddable into Hilbert space if and only if for all $R>0$ and $\varepsilon>0$ there exists a Hilbert space $\mathcal{H}$ and maps $\xi^i:X_i\to\mathcal{H}$, $x\mapsto \xi^i_x$, such that
		\begin{enumerate}[(1)]
			\item $\|\xi^i_x\| = 1$ for all $x\in X_i$ and $i\in I$;
			\item $\sup\left\{|1-\langle \xi^i_x,\xi^i_y\rangle|:d(x,y)\leq R,x,y\in X\right\}\leq\varepsilon$;
			\item for any $\hat\varepsilon>0$, there exists $S>0$, such that
			$$\sup\left\{|\langle \xi^i_x,\xi^i_y\rangle|:d(x,y)\geq S,x,y\in X_i\right\}<\hat\varepsilon$$
			for all $i\in I$.\hfill$\square$
		\end{enumerate}
	\end{Lem}
	
	\begin{proof}[Proof of Proposition \ref{extension}]
		
		Here we just sketch a modified proof based on  \cite[Theorem 4.1]{DG2003}.
		Given $R>0$, choose a sufficiently large $S_0$ such that $d(N_i,N_j)>8R$ for all $i,j>S_0$. As $N=\bigsqcup_{n\in\IN}N_n$ is coarsely embeddable at infinity and $Q=\bigsqcup_{n\in\IN}N_n$ has Property A at infinity, there exists $N_R>S_0$ such that the sequence of all bounded subsets of $\bigsqcup_{n\geq N_R}N_n$ with diameter no more than $8R$ is coarsely embeddable and the sequence of all bounded subsets of $\bigsqcup_{n\geq N_R}Q_n$ with diameter no more than $2R$ has Property A.
		
		For $n\geq N_R$, any bounded subset of $G_n$ with diameter no more than $R$ is isometric to a subset of $B_{G_n}(e,R)$, where $e$ is the unity of the group $G_n$. We claim that the sequence $\{B_{G_n}(e,R)\}_{n\geq N_R}$ is coarsely embeddable by using Lemma \ref{coarse embeddable}. Namely, for any $\varepsilon>0$ and $R_0>0$, we just need to show that there exist a sequence of $\xi^n:B_{G_n}(e,R)\to\ell^2(B_{Q_n}(e,R),\mathcal{H}^n)$ such that $\|\xi^n_g\|=1$ for all $g\in G_n$, and $\xi^n$ satisfying the following conditions
		\begin{itemize}
			\item[I.]$|1-\langle \xi^n_g,\xi^n_h\rangle|\leq\varepsilon$ for all $d_{G_n}(g,h)\leq R_0$;
			\item[II.]for any $\hat\varepsilon>0$, there exists $S>0$ such that $\sup\{|\langle \xi^n_g,\xi^n_h\rangle|\,|\,d_{G_n}(g,h)\geq S,g,h\in B_{G_n}(e,R)\}<\hat\varepsilon$ for all $n\geq N_R$.
		\end{itemize}
		
		Since the sequence $\{B_{Q_n}(e,R)\}_{n\geq N_R}$ has property A, according to Example \ref{Property A at infinity}, for given $\varepsilon$ and $R_0$, there exists a sequence of maps $\lambda^n:B_{Q_n}(e,R)\to\ell^2(B_{Q_n}(e,R))$ and $S_Q>0$ such that $\|\lambda^n_p\|=1$ for all $p\in B_{Q_n}(e,R)$ and $n\geq N_R$, and satisfying the following conditions
		\begin{itemize}
			\item[$1_Q$.]$|1-\langle \lambda^n_p,\lambda^n_q\rangle|\leq\frac{\varepsilon}2$ for all $d_{Q_n}(p,q)\leq R_0$ and $n\geq N_R$;
			\item[$2_Q$.]$\supp(\lambda^n_p)\subseteq B_{Q_n}(p,S_Q)\cap B_{Q_n}(e,R)$ for all $n\geq N_R$.
		\end{itemize}
		
		Since the sequence $\{B_{N_n}(e,4R)\}_{n\geq N_R}$ is coarsely embeddable, according to Lemma \ref{coarse embeddable}, for given $\varepsilon$ and $R_0$, there exists a sequence of Hilbert space valued functions $\zeta^n:B_{N_n}(e,4R)\to  H^n$ satisfying that $\|\zeta^n_a\|=1$ for all $a\in N_n$ and $n\geq N_R$, and such that
		\begin{itemize}
			\item[$1_N$.]$|1-\langle \zeta^n_a,\zeta^n_b\rangle|\leq\frac{\varepsilon}2$ for all $d_{N_n}(a,b)\leq 2S_Q+R_0$;
			\item[$2_N$.]for any $\hat\varepsilon>0$, there exists $S_N>0$, such that $\sup\{|\langle \zeta^n_a,\zeta^n_b\rangle|\,|\,d_{N_n}(a,b)\geq S_G,a,b\in B_{N_n}(e,4R)\}<\hat\varepsilon$ for all $n\geq N_R$.
			$\sup\{|\langle \zeta^n_a,\zeta^n_b\rangle|\,|\,d_{N_n}(a,b)\geq S_G,a,b\in B_{N_n}(e,4R)\}$ goes to $0$ uniformly by $n$ as $S_N$ goes to $\infty$.
		\end{itemize}
		
		Choose a set-theoretic section $\sigma_n:Q_n\to G_n$ of the quotient map $\pi_n:G_n\to Q_n$ with the property that
		$$d_{Q_n}(p, e) = d_{G_n}(\sigma_n(p), e) \mbox{ for all }p\in Q_n.$$
		For any $g\in G_n$, $p\in Q_n$ and $n\in\IN$, as in \cite[Theorem 4.1]{DG2003}, define $\eta_n:G_n\times Q_n\to N_n$ by
		$$\eta_n(g,p)=\sigma_n(p)^{-1}g\sigma_n(\pi_n(g)^{-1}p).$$
		
		Then, define $\xi^n:B_{G_n}(e,R)\to\ell^2(B_{Q_n}(e,R),\mathcal{H}^n)$ by
		$$\xi^n_g(p)=\lambda^n_{\pi_n(g)}(p)\zeta^n_{\eta_n(g,p)}.$$
		
		One can verify  that $\xi^n$ satisfying the required properties I and II. As all parameters in condition ($1_N$), ($2_N$), ($1_Q$) and ($2_Q$) does not depend on $R$. We complete the proof.

	\end{proof}
	
	In the rest of this section, we will recall the notions of warped metric and warped cones introduced by J. Roe \cite{R2005}. A warped cone is constructed from the action of a finitely generated group $\Gamma$ on a compact metric space.  It establishes the bridge linking dynamic systems and coarse geometry. The warped cones have been widely used to construct examples of metric spaces which do not admit a coarse embedding into Hilbert space.  We will use warped cones to provide more examples.
	
	\begin{Def}[\cite{R2005,S2018}]
		Let $(Y,d_Y)$ be a compact metric space and $(\Gamma,S)$ be a group with a finite generating set $S$ which acts by homeomorphisms on $Y$. The intrinsic open cone over $Y$ is $\O Y=Y\times\IR_+$ equipped with the metric defined by
		$$d_{\O Y}((y_1,t_1)(y_2,t_2))=|t_1-t_2|+\frac{\min(t_1,t_2)d_Y(y_1,y_2)}{\diam(Y)},$$
		and the obvious extension of the $\Gamma$-action on $Y$.
		
		The warped cone of $Y$, denoted by $\O_{\Gamma}Y$, is the intrinsic open cone $\O Y=Y$ equipped with the warped metric $d_{\Gamma}$, which is the greatest metric that satisfies the inequalities
		$$d_{\Gamma}(x,y)\leq d(x,y),\qquad d_{\Gamma}(x,sx)\leq 1$$
		for any $x\in \O Y$ and $s\in S$.
	\end{Def}
	
	\begin{Def}
		An action $\Gamma\curvearrowright Y$ is linearisable in a Banach space $E$ if there exists an isometric representation of $\Gamma$ on $E$ and a bi-Lipschitz equivariant embedding $Y\to E$.
	\end{Def}
	
	Let $q:X\to Y$ be a surjective map between two metric spaces. $q$ is said to be \emph{asymptotically faithful} if for every $R>0$, there exists $K\subseteq Y$ such that for any $y\in Y\backslash K$ and $x\in q^{-1}(y)$, the map $q$ restricts to an isometry from $B(x,R)$ to $B(y,R)$.  For any $1\leq p<\infty$ and Banach spaces $E_1$, $E_2$, denote by $E_1\oplus_pE_2$ the $\ell^p$-direct sum:
	$$E_1\oplus_pE_2=\{(x,y)\in E_1\times E_2\mid x\in E_1,y\in E_2\}$$
	with the norm
	$$\|(x,y)\|=\left(\|x\|^p+\|y\|^p\right)^{1/p}.$$
	
	We conclude this section by the following theorem.
	
	\begin{Pro}\label{warped cone}
		Let $1\leq p<\infty$. Assume that the action $\Gamma\acton Y$ is free and linearizable in $\ell^p$ and $\Gamma$ admits a coarse embedding into $\ell^p$, then the warped cone $\O_{\Gamma}Y$ admits a coarse embedding at infinity  into $\ell^p$.
	\end{Pro}
	
	We remark that it has been proved independently in \cite{SW2020} and \cite{WW2017} that if the action $\Gamma\acton Y$ is free and linearizable and $\Gamma$ is a-T-menable, then $\O_{\Gamma}Y$ admits a fibred coarse embedding into Hilbert space.
	
	\begin{proof}[Proof of Theorem \ref{warped cone}]
		We define the twisted metric $d^1$ on $\Gamma\times\O Y$ to be the largest metric such that
		\begin{equation*}\label{eq1}d^1((\gamma,x),(\gamma,y))\leq d_{\O Y}(\gamma x,\gamma y)\quad\mbox{and}\quad d^1((\gamma,x),(s\gamma,x))=1\end{equation*}
		for all $s\in S\backslash\{e\}$, $\gamma\in\Gamma$ and $x,y\in\O Y$. Define the action $\Gamma\acton\Gamma\times\O Y$ by the formula
		$$\gamma(\eta,x)=(\eta\gamma^{-1},\gamma x).$$
		Note that the action on $\Gamma\times\O Y$ is isometric. Consider the quotient map $q:\Gamma\times\O Y\to\O Y$ defined by $(\gamma,x)\mapsto \gamma x$. By \cite[Proposition 3.7]{SW2020}, we have that the warped metric $d_{\Gamma}$ is equal to the quotient metric of $d^1$ under the quotient map $q:\Gamma\times\O Y\to\O Y$. Moreover, as the group action $\Gamma\acton\O Y$ is free, we have $q:\Gamma\times\O Y\to\O_{\Gamma}Y$ is asymptotically faithful by \cite[Proposition 3.10]{SW2020}.
		
		By a similar argument in Corollary \ref{Galois cover}, it suffices to prove that $(\Gamma\times\O Y,d^1)$ is coarsely embeddable. Note that the equivariant bi-Lipschitz embedding $Y\to \ell^p$ can be extended to an equivariant bi-Lipschitz embedding $f_{\O Y}:\O Y\to \IR\oplus_p \ell^p$. Similarly, $\Gamma$ also admits a coarse embedding $f_{\Gamma}:\Gamma\to \ell^p$. Hence, it is easy to check that the map
		$$(\Gamma\times\O Y,d^1)\to \ell^p\oplus_p\IR\oplus_p \ell^p$$
		defined by
		$$(\gamma,r,y)\mapsto(f_{\Gamma}(\gamma),f_{\O Y}(r,y))$$
		for all $\gamma\in\Gamma$, $(r,y)\in\O Y$, is a coarse embedding. Note that $\ell^p\oplus_p\IR\oplus\ell^p$ is isometric to $\ell^p$. Thus we complete the proof.
	\end{proof}
	
	\section{Generalized expanders at infinity }
	
    In this section, we will first introduce a notion of generalized expanders at infinity, which is the obstruction of coarse embedding at infinity into Hilbert space. Then, we introduced a notion of containment of an expander at infinity regarded as a special case of generalized expanders at infinity.
	
	In \cite{Tessera}, R. Tessera introduced a notion of generalized expanders which obstructs to coarse embedding into Hilbert space. Here we introduce the following notion of generalized expanders at infinity.
		\begin{Def}
		A metric space $(X,d)$ is called a {\em generalized expander at infinity} if there exists a constant $c>0$,  and two sequences of positive numbers $\{r_m\}_{m \in \mathbb{N}}$ and $\{R_m\}_{m \in \mathbb{N}}$ increasing to infinity, such that for every bounded subset $K\subset X$ and every $m \in \mathbb{N}$, there exists a subset $C_{K,m}\subset X/K$ with $\diam (C_{K,m})\leq r_m$ and a probability measure $\mu$ on $C_{K,m}\times C_{K,m}$ such that
		\begin{enumerate}[(1)]
			\item $\mu(x,y)=\mu(y,x);$
			\item $\mu(x,y)=0$ if $d(x,y)\leq R_m;$
			\item for every map $f$ from $C_{K,m}$ to a Hilbert space $H$ satisfying $$\|f(x)-f(y)\|\leq d(x,y)$$ for all $x, y \in C_{K,m}$, the inequality
			$$\sum_{x,y \in C_{K,m}}\|f(x)-f(y)\|^2\mu(x,y)\leq c$$
			holds.
		\end{enumerate}
	\end{Def}

	Now, let us show that the generalized expanders at infinity is an obstruction to the coarse embeddability at infinity. We have the following result.
	\begin{Thm}\label{main result}
		A  metric space $X$ does not coarsely embed at infinity into Hilbert space if and only if $X$ coarsely contains a generalized expander at infinity.
	\end{Thm}

	To prove Theorem \ref{main result}, we need the following lemma.
	
	\begin{Lem}\label{lemma}
		Let $X$ be a metric space, and let $\rho:\mathbb{R}_{+}\rightarrow \mathbb{R}_{+}$ be a non-decreasing function satisfying $\lim\limits_{t\rightarrow \infty}\rho(t)=\infty$. Assume that there exists a sequence of increasing positive numbers $\{R_n\}_{n\in\mathbb{N}}$ with $\lim\limits_{n \to \infty} R_n=\infty$, such that for all $i\geq R_n$, there exists a bounded subset $K_{n,i}\subseteq X$ satisfying that for any $C\subseteq X/K_{n,i}$ with $\diam (C)\leq i$, there is a map $f^C_{n,i}:C\rightarrow H$ from $C$ to a Hilbert space $H$, such that
		\begin{itemize}
			\item $\|f^C_{n,i}(x)-f^C_{n,i}(y)\|\leq \rho(d(x,y));$
			\item $\|f^C_{n,i}(x)-f^C_{n,i}(y)\|\geq n$ for $x,y \in C$ with $d(x,y)\geq R_n$.
		\end{itemize}
		Then the space $X$ admits a coarse embedding into Hilbert space at infinity.
	\end{Lem}
	
	\begin{proof}
		For any $r>0$, we define a bounded subset $K_r$ of $X$ to be
		$$K_r=\bigcup_{k_n\leq r} K_{n,r},$$
		where each $K_{n,r}\subset X$ is the bounded subset with respect to $n,r$ above. Then for  any $C \subset X/K_r$ with $\diam (C)\leq r$ and any $n\in \mathbb{N}$ with $R_n\leq r$, there exists a map $f^C_{n,r}:C\rightarrow H$ such that
		\begin{itemize}
			\item [(1)] $\|f^C_{n,r}(x)-f^C_{n,r}(y)\|\leq \rho(d(x,y));$
			\item [(2)] $\|f^C_{n,r}(x)-f^C_{n,r}(y)\|\geq n$ for $x,y \in C$ with $d(x,y)\geq R_n$.
		\end{itemize}
		We define a map $f:C\rightarrow H$ by
		$$f(x)= \sum_{R_n\leq r}\frac{1}{n}\left(f^C_{n,r}(x)\right).$$
		For any $x,y \in C$, we have
		$$\sqrt{\sharp\{n\in \mathbb{N}\mid R_n\leq d(x,y)\}}\leq \|f(x)-f(y)\|\leq \sqrt{\sum_{{R_n\leq r}}\frac{1}{n^2}}\rho(d(x,y))< \sqrt{\sum_{n\in\mathbb{N}}\frac{1}{n^2}}\rho(d(x,y)).$$
		Denote
		$$\rho_{-}(t)=\sqrt{\sharp\{n\in \mathbb{N}\mid R_n\leq t\}} \mbox{~~~~~and~~~~~}\rho_{+}(t)=\sqrt{\sum_{n\in\mathbb{N}}\frac{1}{n^2}}\rho(t).$$
		Note that $\sharp\{n\in \mathbb{N}\mid k_n\leq d(x,y)\}$ increases to infinity as $t\rightarrow \infty$. Therefore, the function $\rho_{-}(t)$ is non-decreasing, and $\lim\limits_{t \to \infty}\rho_{-}(t)=\infty$.
		
		Consequently, for any $r>0$, there exists a bounded subset $K_r$ as defined above, such that for any $C\subset X/K_r$ with $\diam (C)\leq r$, there is a map $f:C\rightarrow H$ such that
		$$\rho_{-}(d(x,y))\leq \|f(x)-f(y)\|\leq \rho_{+}(d(x,y)).$$
		This finishes the proof.
	\end{proof}
	
	We are now ready to prove Theorem \ref{main result}.
	
	\begin{proof}[Proof of Theorem \ref{main result}]
		Let $X$ be a metric space coarsely containing a family of generalized expanders at infinity. For the sake of contradiction, we assume that $X$ admits a coarse embedding at infinity into Hilbert space , i.e., there exist two  non-decreasing functions $\rho_{-}, \rho_{+}: \mathbb{R}_{+}\rightarrow \mathbb{R}_{+}$ with $\lim\limits_{t\rightarrow \infty}\rho_{-}(t)=\infty$ such that for every $r>0$, there exists a bounded subset $K_r\subset X$ such that for any $C\subset X/K$ with $\diam (C)\leq r$, there exists a map  $f:C\rightarrow H$ from $C$ to a Hilbert space $H$, such that
		$$\rho_{-}(d(x,y))\leq \|f(x)-f(y)\|\leq \rho_{+}(d(x,y))$$ for any $x,y\in C$.
		Note that we can assume $\rho_{+}(d(x,y))=d(x,y)$ by restricting the embedding to a discrete net in $X$ and rescaling the image. Since $X$ coarsely contains a family of generalized expanders, we have a constant $c>0$ and two sequences of increasing positive numbers $r_m$ and $R_m$ and a sequence of subset $C_{m}$ with a probability measure $\nu$ on $C_m\times C_m$ satisfying that $\diam(C_m)\leq r_m$ and $\nu(\{(x, y)\in C_m \times C_m: d(x,y)\leq R_m\})=0$.
		Furthermore, we have that
		$$c\geq \sum_{x,y\in C}\|f(x)-f(y)\|^2\mu(x,y)\geq \sum_{x,y\in C}\rho_{-}(x,y)^2\mu(x,y) \geq \rho_{-}(R_m)^2.$$
		This is a contradiction.
		
		To prove the converse, we assume that the metric space X does not coarsely embed at infinity into a Hilbert space. Let $\rho_+(t): \mathbb{R}_+\to \mathbb{R}_+$ be a non-decreasing function with $\rho_+(t)\to \infty$ as $t \to \infty$ and $\{R_m\}_m$ be an increasing sequence of positive numbers with $R_m\to \infty$ as $m \to \infty$. By Lemma \ref{lemma}, there exists a constant $N\in \mathbb{N}$ such that for all $m>0$, there exists $r_m\geq R_m$ such that for any bounded subset $K$ of $X$, there exists $C_{K,m}\in X/K$ with $\diam (C_{K,m})\leq r_m$, for every map $f: C_{K,m}\rightarrow H$ satisfying
		$$\|f(x)-f(y)\|\leq \rho_+(d(x,y)),$$
		there exist points $x,y\in C_{K,m}$ with $d(x,y)\geq m$ such that
		$$\|f(x)-f(y)\|\leq N.$$
		We will construct a generalized expander at infinity using the sequence $C_{K,m}$ defined above following the arguments in \cite[Theorem 5.7.3]{NY2012} or \cite{Tessera}.
		
		Let $\{r_m\}$ and $\{R_m\}$ be the sequences of positive constants as above. For each bounded subset $K \subseteq X$ and $m>0$.  Denote by $A_{K,m}$ the set of functions $\psi:C_{K,m}\times C_{K,m}\rightarrow \mathbb{R}_{+}$ satisfying
		\begin{itemize}
			\item $\psi(x,x')\leq \rho_+(d(x,x'))$, for all $ x,x' \in C_{K.m}$.
			\item $\psi(y,y')\geq N$, for all $y,y' \in C_{K,m}$ with $d(y,y')\geq R_m$.
		\end{itemize}

		Next, we will define a measure $\mu$ on each set $C_{K,m}\times C_ {K,m}$, such that $\mu(\{(x, y)\})=0$ for all $d(x,y)\geq R_m$ and for every map $f$ from $C_{K,m}$ to a Hilbert space $H$ satisfying
		$$\|f(x)-f(y)\|\leq d(x,y),~~~~\forall~x, y \in C_{K,m},$$
		the inequality
		$$\sum_{x,y \in C_{K,m}}\|f(x)-f(y)\|^2\mu(x,y)\leq c$$
		holds.
		
		Let $B_{K,m}$ be the set of functions $\psi:C_{K,m}\times C_{K,m} \rightarrow \mathbb{R}_{+}$ which  are of the form $$\psi(x,y)=\|f(x)-f(y)\|$$
		for some $f:C_{K,m}\rightarrow H$ satisfying
		$$\|f(x)-f(y)\|\leq d(x,y),~~~~~\forall~x,y \in C_{K,m}.$$
		The sets $A_{K,m}$ and $B_{K,m}$ are convex subsets of $\ell^2(C_{K,m}\times C_{K,m})$ and they are disjoint by Lemma \ref{lemma}. By the Hahn–Banach theorem, there exists a functional $\phi$ separating $A_{K,m}$ and $B_{K,m}$. By Riesz representation theorem, there exists a fixed vector $v\in \ell^2(C_{K,m}\times C_{K,m})$  such that
		\begin{itemize}
			\item $\langle v,\psi\rangle\geq 0$, for all $\psi\in A_{K,m}$;
			\item $\langle v,\psi\rangle\leq 0$, for all $\psi\in B_{K,m}$.
		\end{itemize}
		Let $v_{+}(x,y)=\max\{v(x,y),0\}$ and $v_{-}(x,y)=\max\{-v(x,y),0\}$ for all $x,y \in C_{K,m}$. Define
		$$	
		\psi_1(x,y)=\begin{cases}
		N^2  & \mbox{if } v(x,y)>0 \mbox{ and } d(x,y)\geq R_m\\
		0  & \mbox{ if } v(x,y)>0 \mbox{ and } d(x,y)< R_m\\
		\rho_+^2(d(x,y)) & \mbox{otherwise}.
		\end{cases}
		$$
		Then $\psi_1\in A_r$ and $\langle v,\psi_1\rangle\geq 0$ imply
		$$\sum_{x,y\in C_{K,m}}\rho_+^2(d(x,y))v_{-}(x,y)\leq N^2\sum_{d(x,y)\geq k}v_{+}(x,y)\leq N^2\sum_{d(x,y)\geq k}v_{+}(x,y).$$
		Now, consider $\psi_2(x,y)=\|f(x)-f(y)\|^2$ for some map $f: C_{K,m} \rightarrow H$ satisfying $\|f(x)-f(y)\|\leq \rho_{+}(d(x,y))$. Then $\langle v,\psi_2\rangle\leq 0$ implies
		$$\sum_{x,y\in C_{K,m}}\|f(x)-f(y)\|^2v_{+}(x,y)\leq \sum_{x,y\in C_{K,m}}\rho_+^2(d(x,y))v_{-}(x,y).$$
		
		Define
		$$
		\displaystyle \mu(x,y)=\begin{cases}
		\frac{v_{+}(x,y)}{\sum_{d(x,y)\geq k}  v_{+}(x,y)} & \mbox{ if } d(x,y)\geq R_m \\
		0 & \mbox{otherwise}.
		\end{cases}
		$$
		It is not hard to verify that the theorem follows.
	\end{proof}

	In the rest of this section, we will introduce a notion of expanders at infinity as a special case of generalized expanders at infinity.	We first recall the definition of discrete Laplace operator on a finite graph. A graph $X=(V,E)$, where $V$ is the set of vertices and $E$ is the set of edges, is equipped with the graph metric. For $x,y\in V$, write $x\sim y$ if there is an edge connecting them. The degree of a vertex $x$ is the number of the set $\{y\in V: x\sim y\}$. Denote $|V|$ to be the number of the vertices of $X$.
	\begin{Def}
		Let $X=(V,E)$ be a finite connected graph. The Laplace operator on  $V$ is a linear operator $\bigtriangleup_X$ defined by $$\bigtriangleup:\ell^2(V)\to \ell^2(V), \quad (\bigtriangleup_Xu)(x)=\sum_{x\sim y}(u(x)-u(y)).$$
	\end{Def}
	It is well-known that $\Delta_X$ is a positive operator. Denote by $\lambda_1(X)$ the first non-zero eigenvalue of $\bigtriangleup_X$. We propose the following notion of containment of an expander at infinity. Roughly speaking,  in these spaces, there exists a sequence of connected subgraphs around infinity, whose number of the vertices  increase to infinity and the first non-zero eigenvalue has a uniformly positive lower bound.
	
	\begin{Def}\label{expander at infty}
		Let $\{ X_n=(V_n, E_n)\}_{n\in \mathbb{N}}$ be a sequence of finite connected graphs with bounded degree, and $|V_n|\to \infty$. The sequence $\{X_n\}_{n \in \mathbb{N}}$ is said to {\em contain an expander at infinity} if there exists a constant $c>0$ and a sequence of positive numbers $\{r_m\}_{m\in \mathbb{N}}$ increasing to infinity, such that for any integers $N$ and $m$, there is a connected subgraph $C_{N,m} \subset X_n$ for some $n\geq N$,
		such that the sequence of subgraphs $\{C_{N,m}\}_{N,m}$ satisfies the following:
		\begin{itemize}
			\item $\diam (C_{N, m})\leq r_m$ for all $N,m$;
			\item for each fixed $N$, $|C_{N,m}|$ increases to infinity as $m$ goes to infinity;
			\item $\lambda_1(C_{N,m})\geq c,$ for all $N, m$.
		\end{itemize}
	\end{Def}
	We remark here that our definition of expanders at infinity is different from the concept of expanders. In fact, expander maybe not contains expanders at
	infinity. The first example of expanders was constructed by G. Margulis using the box space of a residually finite group with Kazhdan's Property (T). A typical example of group with Property (T) is $SL_3(\mathbb{Z})$. E. Guentner, N. Higson and S. Weinberger \cite{GHW-LinearGroup} showed that the linear group $SL_3(\mathbb{Z})$ can be coarsely embedded into Hilbert space. However, the following Theorem \ref{thm-expander-and-expander-at-infty} implies that a group can not be coarsely embedded into Hilbert space whenever it has a box space which contains expanders at infinity.
	
	\begin{Thm}\label{thm-expander-and-expander-at-infty}
		A  finitely generated residually finite group $\Gamma$ contains an expander if and only if one or every box space of the group $\Gamma$ contains an expander at infinity.
	\end{Thm}
	\begin{proof}
		Let $\Gamma$ be a residually finite group with a nested sequence of normal subgroups $N_n$ of finite index such that $\cap N_n=\{e\}$. We consider the box space associated with the filtration $\{N_n\}_{n\in \mathbb{N}}$.
		
		``$\Longrightarrow.$''
		Assume that the group $\Gamma$ contains a sequence of finite and connected subgraphs $\{X_n=(V_n,E_n)\}$ such that
		\begin{itemize}
			\item $|X_n|\to \infty$ as $t \to \infty$;
			\item there exists a constant $c>0$ such that $\lambda_1(X_n)\geq c$ for all $n \in \mathbb{N}$.
		\end{itemize}
		We will show that the box space associated with $\{N_n\}_{n \in \mathbb{N}}$ contains expanders at infinity. Let $r_m=|V_m|$. For each $n, N>0$, we fix an $n>N$ such that $\pi_n: G\to G/N_n$ is an isometry. Let $C_{N,m}=\pi_n(V_m)$. It is obvious that ${\rm diam}(C_{N,m})\leq r_m$ and $\lambda_1(C_{N,m})>c$. Therefore, we have that the box space associated with $\{N_n\}_{n \in \mathbb{N}}$ contains expanders at infinity.

		``$\Longleftarrow.$'' Assume that the box space associated to $\{N_n\}_{n \in \mathbb{N}}$ contains expanders at infinity. Let $r_m$ be a sequence of increasing sequence and $c>0$. For each $r_m$ and $n>N$ with $\pi_n: \Gamma \to \Gamma/N_n$ is isometric on each $r_m$-ball, there exists a finite connected subgraph $C_{N,m}\subset \Gamma/N_n$ such that ${\rm diam}(C_{N,m}) \leq r_m$. Consider $X_m=\pi^{-1}(C_{N,m})$. Then we have a sequence of finite subset $\{X_m\}_{m \in \mathbb{N}}$. It is obvious that $\{X_m\}_{m \in \mathbb{N}}$ is an expander.
	\end{proof}
	
	In the following, we prove that the containment of an expander at infinity is an obstruction to the coarse embeddability at infinity into Hilbert space.
	
	\begin{Thm}
			Let $\{X_n\}_{n\in \mathbb{N}}$ be a family of connected finite graphs. If $\{X_n\}_{n\in \mathbb{N}}$ contains an expander at infinity, then the coarse disjoint union $X=\sqcup_{n\in \mathbb{N}} X_n$ does not coarsely embed at infinity into Hilbert space.
	\end{Thm}
	\begin{proof}
		For the sake of contradiction, we assume that $X$ admits a coarse embedding at infinity into Hilbert space. More precisely,  there exist two non-decreasing functions $\rho_{-},\rho_{+}:\mathbb{R}_{+}\to \mathbb{R}_{+}$ with $\lim\limits_{t\rightarrow \infty}\rho_{-}(t)=\infty$ such that for any $r>0$, there exists a integer $N_r$, such that for all $n>N_r$ and $C_r\subset X_n$ with $\diam (C)\leq r$, there exists a map $f_C$ from $C$ to a Hilbert space $H$ such that
		$$\rho_{-}(d(x,y))\leq \|f(x)-f(y)\|\leq \rho_{+}(d(x,y))$$
		for any $x,y\in C$.
		
		For each $C$, consider the tensor product Hilbert space $\ell^2(C)\otimes H$ and the Laplace operator $\bigtriangleup_{C}=\bigtriangleup_{C}\otimes Id_H$.
		
		Replacing each $f_C$ with $f_C-\frac{1}{|C|}\sum_{x\in C}f_C(x)$, since $X$ contains expanders at infinity, there exists a constant $c>0$ such that for any $r>0$, there is a connected subgraph $C$ contained in some finite graph of $\{X_n\}_{n\geq N_r}$ with $\diam (C)\leq r$ and $|C|$ increasing to infinity as $r$ tends to $\infty$ such that
		$$c\langle f_C, f_C \rangle\leq \langle f_C, \bigtriangleup_C f_C \rangle .$$
		Since $\{X_n\}_{n\in \mathbb{N}}$ has upper bounded degree, denoted by $k$,	then we have
		$$c\sum_{x\in C}\|f_C(x)\|^2\leq \sum_{x\sim y\in C}\|f_C(x)-f_C(y)\|^2\leq k|C|\rho_{+}(1)^2.$$
		It follows that at least half of the points in $C$ 	satisfy $\|f_C(x)\|^2\leq \frac{2k\rho_{+}(1)^2}{c}$. Following the above, we can choose a sequence of bounded subset $\{C_m\}_{m\in \mathbb{N}}$ such that $|C_m|\to \infty$ as $m \to \infty$.
		Combined with bounded degree, this contradicts the existence of $\rho_{-}$. Therefore, the coarse disjoint union $X=\sqcup_{n\in \mathbb{N}} X_n$ does not embed coarsely at infinity into Hilbert space. 		
	\end{proof}

	Finally, we show that containing an expander at infinity provides typical examples of generalized expanders at infinity.
\begin{Pro}
	The coarse disjoint union of a family of graphs which contains an expander at infinity is a generalized expander at infinity.
\end{Pro}
\begin{proof}
	Assume that $X=\{X_n\}_{n\in \mathbb{N}}$ contains an expander at infinity. Then we have a constant $c>0$ and a sequence of strictly increasing positive numbers $\{r_m\}_{m \in \mathbb{N}}$. For any integer $N$ and $m$, there is a connected subgraph $C_{N,m} \subset X_n$ for some $n\geq N$,
	such that the sequence of subgraphs $\{C_{N,m}\}_{N,m}$ satisfies the following:
	\begin{itemize}
		\item $\diam (C_{N, m})\leq r_m$ for all $N,m$;
		\item for each fixed $N$, $|C_{N,m}|$ increases to infinity as $m$ goes to $\infty$;
		\item $\lambda_1(C_{N,m})\geq c,$ for all $N, m$.
	\end{itemize}

For any bounded subset $K\subset X$, we can assume that $K\subset \bigcup_{i=1}^N X_i$ for some $N$. Then $C_{N,m}\subset X \setminus K$. Since $\lambda_1(C_{N,m})\geq c$, we have that
$$\sum_{x\in C_{N,m}}\|f(x)\|^2\leq \frac{1}{c}\sum_{x\sim y\in C_{N,m}}\|f(x)-f(y)\|^2,$$
where for every function $f:C_{N,m}\to H$ satisfies $\sum_{x\in C_{N,m}}f(x)=0.$

Next, we will find the constant $R_m$ for each $m>0$. Let $C\subset X$ be a subgraph with $\lambda_1(C)\leq c$. For such $f$, we have
\begin{equation*}
\begin{aligned}
&\sum_{x,y\in C}\|f(x)-f(y)\|^2\\
=& \sum_{x\in C}\sum_{y\in C}(\|f(x)\|^2+\|f(y)\|^2)-\langle\sum_{x\in C}f(x),\sum_{y\in C}f(y)\rangle\\
=&2|C|\sum_{x\in C}\|f(x)\|^2.
\end{aligned}
\end{equation*}
Hence we get
$$\frac{1}{|C|^2}\sum_{x,y\in C}\|f(x)-f(y)\|^2\leq \frac{1}{c|C|}\sum_{x\in C}\sum_{y\sim x}\|f(x)-f(y)\|^2.$$
Let $\mu$ to be the uniform measure on $C\times C$, namely, $\mu(x,y)=\frac{1}{|C|^2}$ for all $(x,y)\in C\times C$.
Then
$$\sum_{x,y\in C}\|f(x)-f(y)\|^2\mu(x,y)\leq \frac{1}{c|C|}\sum_{x\in C}\sum_{y\sim x}\|f(x)-f(y)\|^2.$$
If $f$ is $1$-Lipschitz, we have
	$$\sum_{x,y\in C}\|f(x)-f(y)\|^2\mu(x,y)\leq \frac{k_0}{c},$$
where $k_0$ is the upper bound on the degree of each vertex.

Let $R=\log_{k_0}(\frac{|C|}{2})$. Then the number of elements $(x,y)\in C\times C$ satisfying $d(x,y)\leq R$ is at most $k_0^{R}|C|$. Then,  the uniform measure satisfies
$$\sum_{d(x,y)\geq R}\mu(x,y)\geq \frac{1}{2}.$$

By the definition of the expanders at infinity, for each $r_m$ and each $K=\bigcup^N_{i=1}X_i$ for some $i$, we choose $C=C_{N,m}$.  Then the $R_m$ can be defined as above.

Since $|C_{N,m}|\to \infty$  as $m\to \infty$, it follows that $R_m$ tends to infinity as $m$ goes to infinity. Denote by $\nu$ the normalized restriction of $\mu$ to the complement of the $R_m$-neighborhood of the diagonal. Then
$$\sum_{x,y\in C_{N,m}}\|f(x)-f(y)\|^2\nu(x,y)\leq \frac{k_0}{c}.$$
\end{proof}

	\section{The general setting}
	In this section, we consider the obstructions to the coarse embeddability at infinity into general
	metric space, and introduce a notion of generalized expanders at infinity in  a general setting,  which is an obstruction to the coarse embedding at infinity into certain Banach spaces, i.e. $L^p$-spaces, CAT(0)-spaces, uniformly convex Banach spaces.
	
	Let $S$ be a set and $(Y,d)$ a metric space. If $f:S\to Y$ is a map, then we  consider the pull-back pseudo-metric $$\sigma_f(x,y)=d(f(x),f(y))$$ for all $x,y \in S$. Such pseudo-metrics are called $Y$-metrics on $S$. Let $\mathcal{M}$ be a class of metric spaces. A metric on the set $S$ is an $\mathcal{M}$-metric if it is a $Y$-metric for some $Y \in \mathcal{M}$.
	
	\begin{Def} Let $S$ be a set and $1\leq p <\infty$.
		A class of metrics $\{\sigma_i\}_{i\in I}$ on $S$ is called p-convex if the set of $\{\sigma_i^p\}_{i\in I}$ forms a convex cone of the space of functions on $S\times S$.	
	\end{Def}
We say that an $\mathcal{M}$-metric is a Hilbert metric ($\ell^p$-metric,
CAT(0)-metric) if $\mathcal{M}$ is the class of Hilbert spaces ($\ell^p$
p-metric, CAT(0)-metric). Note that the set of squares of Hilbert metrics forms a convex cone of the space of real-valued functions on $S\times S$.
	\begin{Exa}	
		\begin{enumerate}[(1)]	
			\item  For $p \geq 1$, $L^p$-metrics is p-convex.	
			\item  Let $c > 0$ and $1 < p < \infty$. The class $\mathcal{M}_{c,p}$ of $(c, p)$-uniformly convex Banach	spaces, is the class of uniformly convex Banach spaces whose moduli of convexity satisfy $\delta(t) \geq ct^p$. $\mathcal{M}_{c,p}$-metrics is p-convex.		
			\item  CAT(0)-metrics is $2$-convex.	
		\end{enumerate}
	\end{Exa}
	
	Assume that $(X,d)$ is a metric space. A metric $\sigma$ on $X$ is called coarse if there exist two non-decreasing unbounded functions $\rho_{-}, \rho_+$ such that for all $x,y \in X$, $$\rho_{-}(d(x,y))\leq \sigma(x,y)\leq\rho_{+}(d(x,y)).$$
	\begin{Def}
		Let $Y$ be a metric space.  We say a metric space $X$ adimts a coarse embedding at infinity into $Y$   if if there exist two non-decreasing unbounded functions $\rho_{-}, \rho_+$ such that for every $r>0$, there exists a bounded subset $K_r\subset X$ such that for any $C\subset X/K_r$ with $\mathop{diam}C\leq r$, there exists a  coarse $Y$-metric $\sigma _C$ on $C$ with respect to $\rho_{-}, \rho_+$.
	\end{Def}
	
	For a class of metric spaces $\mathcal{M}$, we may consider $(\mathcal{M},p)$-generalized expanders at infinity.
	
	\begin{Def}
	A metric space $(X,d)$ is called a  $(\mathcal{M},p)$-generalized expander at infinity if there exists a $c>0$,  and two sequences of positive numbers $\{r_m\}_{m \in \mathbb{N}}$ and $\{R_m\}_{m \in \mathbb{N}}$ increasing to infinity, such that for every bounded subset $K\subset X$ and every $m \in \mathbb{N}$, there exists a subset $C_{K,m}\subset X/K$ with $\diam (C_{K,m})\leq r_m$ and a probability measure $\mu$ with support on $C_{K,m}\times C_{K,m}$ such that
	\begin{enumerate}[(1)]
		\item $\mu(x,y)=\mu(y,x);$
		\item $\mu(x,y)=0$ if $d(x,y)\leq R_m;$
		\item for every map $f$ from $C_{K,m}$ to a metric space $(Y,\sigma)\in \mathcal{M}$ satisfying $$\sigma(f(x),f(y))\leq d(x,y)$$ for all $x, y \in C_{K,m}$, the inequality
		$$\sum_{x,y \in C}\sigma(f(x),f(y))^p\mu(x,y)\leq c$$
		holds.
	\end{enumerate}
\end{Def}

	Following the proof of Theorem \ref{main result}, we have the following more general result.
	
	\begin{Thm}
		Let $X$ be a metric space and $\mathcal{M}$ be a class of p-convex metric spaces. Then there exists no coarse embedding  at infinity into any space in $\mathcal{M}$ if and only if $X$ coarsely contains a  $(\mathcal{M},p)$-generalized expanders at infinity.
	\end{Thm}

	\vskip 1cm
	\begin{itemize}
		\item[] Jintao Deng\\
		Department of Mathematics,
		University of Waterloo, Waterloo N2L3G1, Canada\\
		E-mail: jintao.deng@uwaterloo.ca

		\item[] Liang Guo\\
		Research Center for Operator Algebras, School of Mathematical Sciences, East China Normal University, Shanghai, 200241, P. R. China.\\
		E-mail: 52205500015@stu.ecnu.edu.cn
		
		\item[] Qin Wang \\
		Research Center for Operator Algebras,  and Shanghai Key Laboratory of Pure Mathematics and Mathematical Practice, School of Mathematical Sciences, East China Normal University, Shanghai, 200241, P. R. China. \\
		E-mail: qwang@math.ecnu.edu.cn
		
		\item[] Yazhou Zhang \\
		Research Center for Operator Algebras, School of Mathematical Sciences, East China Normal University, Shanghai, 200241, P. R. China.\\
		E-mail: 52185500010@stu.ecnu.edu.cn

	\end{itemize}
	
\end{document}